\documentclass[11pt,reqno]{amsart}
\usepackage[percent]{overpic}
\usepackage{amsfonts,amssymb,amsmath,color}
\usepackage{hyperref}

\DeclareMathOperator{\rank}{\mathrm{rank}} \DeclareMathOperator{\supp}{\mathrm{supp}}
\DeclareMathOperator{\sign}{\mathrm{sign}}

\newtheorem{theorem}{Theorem}

\newtheorem{proposition}[theorem]{Proposition}

\newtheorem{remark}[theorem]{Remark}
\newtheorem{example}[theorem]{Example}

\numberwithin{equation}{section} \numberwithin{theorem}{section}

\begin{document}

\title{Pattern Recognition on Oriented Matroids: Subtopes and Decompositions of (Sub)topes}

\author{Andrey O. Matveev}
\email{andrey.o.matveev@gmail.com}

\begin{abstract}
For a symmetric $2t$-cycle in the tope graph of a simple oriented matroid $\mathcal{M}$ on the ground set~$\{1,\ldots,t\}$, where $t$ is {\em even}, we describe decompositions of topes and subtopes of $\mathcal{M}$ with respect to the subtopes corresponding to the edges of the symmetric cycle.
\end{abstract}

\maketitle

\pagestyle{myheadings}

\markboth{PATTERN RECOGNITION ON ORIENTED MATROIDS}{A.O.~MATVEEV}

\thispagestyle{empty}



\section{Introduction}

Let $\mathcal{M}:=(E_t,\mathcal{T})$ be a {\em simple\/} oriented matroid, with set of {\em topes\/}~$\mathcal{T}\subseteq\{1,-1\}^t$, on the {\em ground set\/} $E_t:=[t]:=[1,t]:=\{1,\ldots,t\}$, where $t\geq 4$. See, e.g.,~\cite{BK,BLSWZ,Bo,BS,DRS,K,S,Z} on {\em oriented matroids}. Using a nonstandard terminology,
by ``simple'' we mean that the oriented matroid~$\mathcal{M}$ has no loops, parallel or {\em antiparallel\/} elements. We regard the
set of {\em covectors\/} $\mathcal{L}\subseteq\{1,-1,0\}^t$ of $\mathcal{M}$
as a set of row vectors of the real Euclidean space~$\mathbb{R}^t$.

Let $\mathrm{T}^{(+)}:=(1,\ldots,1)$ denote the $t$-dimensional row vector of all $1$'s; if the oriented matroid $\mathcal{M}$ is {\em acyclic}, then $\mathrm{T}^{(+)}$ is called the {\em positive tope}. For a subset~\mbox{$A\subseteq E_t$,} we denote by~${}_{-A}\mathrm{T}^{(+)}$ the tope $T$ whose {\em negative part\/}~$T^-:=\{e\in E_t\colon T(e)=-1\}$ is the set~$A$. If $s\in E_t$, then we write~${}_{-s}\mathrm{T}^{(+)}$ instead of~${}_{-\{s\}}\mathrm{T}^{(+)}$.

If $T'$ and $T''$ are {\em \textcolor{magenta}{adjacent} \textcolor{blue}{topes}\/} in the {\em tope graph\/} of $\mathcal{M}$ (i.e., the {\em Hamming \textcolor{magenta}{distance}\/} between the {\em \textcolor{blue}{words}\/} $T'$ and $T''$ is \textcolor{magenta}{$1$} or, equivalently, the standard {\em \textcolor{magenta}{scalar product}\/} $\langle T',T''\rangle$ of the {\em \textcolor{blue}{vectors}\/} $T'$ and $T''$ of the space $\mathbb{R}^t$ is \textcolor{magenta}{$t-2$}), then their common {\em \textcolor{magenta}{subtope}\/} $S\in\{1,-1,0\}^t$ defined as
the {\em \textcolor{magenta}{meet}\/}
\begin{equation*}
S:=T'\wedge T''
\end{equation*}
of the \textcolor{blue}{\em elements\/} $T'$ and $T''$ in the {\em big face lattice\/} of the oriented matroid~$\mathcal{M}$, can be interpreted as the {\em \textcolor{magenta}{midpoint}}
\begin{equation}
\label{eq:16}
S=\tfrac{1}{2}(T'+T'')
\end{equation}
of a straight {\em line segment\/} in~$\mathbb{R}^t$, with the {\em \textcolor{blue}{endpoints}\/} $T'$ and $T''$. From this viewpoint, if we let~$\mathbb{S}^{t-1}(r)$ denote the $(t-1)$-dimensional {\em sphere\/} of radius $r$ in $\mathbb{R}^t$, centered at the origin, then~$S$ is the {\em \textcolor{magenta}{point of tangency}\/} of the {\em tangent line\/} to the sphere $\mathbb{S}^{t-1}(\sqrt{t-1})$ passing through the {\em \textcolor{blue}{points}\/} $T'$ and~$T''$ that lie on the sphere~$\mathbb{S}^{t-1}(\sqrt{t})$.

The corresponding edge~$\{\widetilde{T}',\widetilde{T}''\}$ of the {\em hypercube graph\/} $\widetilde{\boldsymbol{H}}(t,2)$ on the vertex set $\{0,1\}^t$, where~$\widetilde{T}':=\tfrac{1}{2}(\mathrm{T}^{(+)}-T')$ and~$\widetilde{T}'':=\tfrac{1}{2}(\mathrm{T}^{(+)}-T'')$, could be labeled by the row~vector $\widetilde{S}:=\tfrac{1}{2}(\widetilde{T}'+\widetilde{T}'')=\tfrac{1}{2}(\mathrm{T}^{(+)}-S)\in\{0,1,\tfrac{1}{2}\}^t$.

A {\em symmetric cycle\/} $\boldsymbol{D}$ in the tope graph of $\mathcal{M}$ is defined to be its $2t$-cycle
with vertex sequence
\begin{equation}
\label{eq:17}
\mathrm{V}(\boldsymbol{D}):=(D^0,D^1,\ldots,D^{2t-1})
\end{equation}
such that
\begin{equation}
\label{eq:18}
D^{k+t}=-D^k\; ,\ \ \ 0\leq k\leq t-1\; .
\end{equation}
The sequence $\mathrm{V}(\boldsymbol{D})$ is a {\em maximal positive basis\/} of the space~$\mathbb{R}^t$,~see~\mbox{\cite[\S{}11.1]{M-PROM}}. We will see that if the cardinality $t$ of the ground set $E_t$ is {\em even}, then the set of subtopes, of the form~(\ref{eq:16}), associated with the edges of the cycle $\boldsymbol{D}$ is also a {\em maximal positive basis\/} of $\mathbb{R}^t$. We describe related (de)composition constructions for topes and subtopes of the oriented matroid~$\mathcal{M}$, and we give a detailed  example to illustrate them. In addition, we consider vertex decompositions in hypercube graphs with respect to the edges of their distinguished symmetric cycles.

\section{Symmetric cycles in tope graphs:\\ vertices, edges and (sub)topes}

Given a symmetric cycle $\boldsymbol{D}:=(D^0,D^1,\ldots,D^{2t-1},D^0)$ in the tope graph of a simple oriented matroid~$\mathcal{M}:=(E_t,\mathcal{T})$, with the sequence~$\mathrm{V}(\boldsymbol{D})$ of its vertices~(\ref{eq:17})(\ref{eq:18}), we denote by~$\mathcal{E}(\boldsymbol{D})$ the edge sequence of the cycle~$\boldsymbol{D}$:
\begin{equation*}
\mathcal{E}(\boldsymbol{D}):=
\bigl(\{D^0,D^1\},\{D^1,D^2\},\ldots,\{D^{2t-2},D^{2t-1}\},\{D^{2t-1},D^0\}\bigr)\; .
\end{equation*}

Let $\mathcal{S}\subset\{1,-1,0\}^t$ denote the set of subtopes of the oriented matroid~$\mathcal{M}$. By means of the map
\begin{equation}
\label{eq:1}
\mathcal{E}(\boldsymbol{D})\to\mathcal{S}\; ,\ \ \ \{D',D''\}\mapsto D'\wedge D''=\tfrac{1}{2}(D'+D'')\; ,
\end{equation}
we associate with the edge sequence of the cycle $\boldsymbol{D}$ the corresponding sequence of subtopes
$\mathrm{S}(\boldsymbol{D})$:
\begin{multline}
\label{eq:2}
\mathrm{S}(\boldsymbol{D}):=
\bigl(S^0:=D^0\wedge D^1,\ S^1:=D^1\wedge D^2,\ \ldots,\\
S^{2t-2}:=D^{2t-2}\wedge D^{2t-1},\ S^{2t-1}:=D^{2t-1}\wedge D^0\bigr)\; ,
\end{multline}
where we have
\begin{equation*}
S^{k+t}=-S^k\; ,\ \ \ 0\leq k\leq t-1\; .
\end{equation*}

Recall that the sequence of topes $(D^0,D^1,\ldots,D^{t-1})$ is a {\em basis\/} of the space~$\mathbb{R}^t$.
Looking at the matrix expression
\begin{equation}
\label{eq:9}
\underbrace{
\left(\begin{smallmatrix}
S^0\\
S^1\\
S^2\vspace{-2mm}\\
\vdots\\
S^{t-3}\\
S^{t-2}\\
S^{t-1}
\end{smallmatrix}\right)}_{\mathbf{W}:=\mathbf{W}(\boldsymbol{D})}=
\frac{1}{2}\;\underbrace{\left(\begin{smallmatrix}
1&1&0&\cdots&0&0&0\\
0\vspace{-2mm}&1&1&\cdots&0&0&0\\
\vdots&\vdots&\vdots&\cdots&\vdots&\vdots&\vdots\\
0&0&0&\cdots&1&1&0\\
0&0&0&\cdots&0&1&1\\
-1&0&0&\cdots&0&0&1
\end{smallmatrix}\right)}_{\boldsymbol{N}(t)\in\mathbb{R}^{t\times t}}\cdot\underbrace{\left(\begin{smallmatrix}
D^0\\
D^1\\
D^2\vspace{-2mm}\\
\vdots\\
D^{t-3}\\
D^{t-2}\\
D^{t-1}
\end{smallmatrix}\right)}_{\mathbf{M}:=\mathbf{M}(\boldsymbol{D})}\; ,
\end{equation}
we see that this relation
\begin{equation*}
\mathbf{W}=\tfrac{1}{2}\boldsymbol{N}(t)\cdot\mathbf{M}
\end{equation*}
implies that
\begin{equation*}
\rank\mathbf{W}=\rank\boldsymbol{N}(t)
=\begin{cases}
t-1\; , & \text{if $t$ is odd\; ,}\\
t\; , & \text{if $t$ is even\; .}
\end{cases}
\end{equation*}
Thus, if $t$ is {\em even}, then the sequence of subtopes $(S^0,S^1,\ldots,S^{t-1})$ is a {\em basis\/} of the space~$\mathbb{R}^t$, the sequence $\mathrm{S}(\boldsymbol{D})$ defined by~(\ref{eq:2}) is a {\em maximal positive basis\/} of $\mathbb{R}^t$, and we have
\begin{equation}
\label{eq:3}
\mathbf{M}=:\boldsymbol{P}(t)\cdot\mathbf{W}\; ,
\end{equation}
where the $(i,j)$th entries of the {\em Toeplitz matrix\/} (see, e.g.,~\cite{N})
\begin{equation*}
\boldsymbol{P}(t):=2\boldsymbol{N}(t)^{-1}\in\mathbb{R}^{t\times t}
\end{equation*}
are
\begin{equation}
\label{eq:8}
\begin{cases}
(-1)^{i+j}\; , & \text{if $i\leq j$\; ,}\\
(-1)^{i+j+1}\; , & \text{if $i>j$\; ,}
\end{cases}
\end{equation}
that is, we have
\begin{equation*}
\underbrace{\left(\begin{smallmatrix}
D^0\\
D^1\\
D^2\vspace{-2mm}\\
\vdots\\
D^{t-3}\\
D^{t-2}\\
D^{t-1}
\end{smallmatrix}\right)}_{\mathbf{M}
}=
\underbrace{\left(\begin{smallmatrix}
1&-1&1&\cdots&-1&1&-1\\
1\vspace{-2mm}&1&-1&\cdots&1&-1&1\\
\vdots&\vdots&\vdots&\cdots&\vdots&\vdots&\vdots\\
1&-1&1&\cdots&1&-1&1\\
-1&1&-1&\cdots&1&1&-1\\
1&-1&1&\cdots&-1&1&1
\end{smallmatrix}\right)}_{\boldsymbol{P}(t)
}\cdot
\underbrace{
\left(\begin{smallmatrix}
S^0\\
S^1\\
S^2\vspace{-2mm}\\
\vdots\\
S^{t-3}\\
S^{t-2}\\
S^{t-1}
\end{smallmatrix}\right)}_{\mathbf{W}
}\; .
\end{equation*}
For example, the matrices $\boldsymbol{P}(4)$ and $\boldsymbol{P}(6)$ are
\begin{equation*}
\boldsymbol{P}(4)=
\left(\begin{smallmatrix}
1&-1&1&-1\\
1&1&-1&1\\
-1&1&1&-1\\
1&-1&1&1
\end{smallmatrix}\right)
\ \ \ \text{and}\ \ \
\boldsymbol{P}(6)=
\left(\begin{smallmatrix}
1&-1&1&-1&1&-1\\
1&1&-1&1&-1&1\\
-1&1&1&-1&1&-1\\
1&-1&1&1&-1&1\\
-1&1&-1&1&1&-1\\
1&-1&1&-1&1&1
\end{smallmatrix}\right)\; .
\end{equation*}

Given a vector~$\boldsymbol{z}:=(z_1,\ldots,z_t)\in\mathbb{R}^t$, we denote its {\em support\/} $\{e\in E_t\colon$ $z_e\neq 0\}$
by~$\supp(\boldsymbol{z})$.

Since the entries of matrices $\boldsymbol{P}(t)$ belong to the set~\mbox{$\{1,-1\}$}, the rows
of these matrices can be viewed as vertices of the {\em hypercube graph\/} $\boldsymbol{H}(t,2)$ on the vertex set~$\{1,-1\}^t$:
\begin{remark}
If $\mathcal{H}:=(E_t,\{1,-1\}^t)$ is the oriented matroid realizable as the {\em arrangement\/} of {\em coordinate hyperplanes\/} in $\mathbb{R}^t$ {\rm(}see, e.g.,~\cite[Example~4.1.4]{BLSWZ}{\rm)}, where $t$ is {\em even}, then the rows $\boldsymbol{P}(t)^i=:P^{i-1}$, $1\leq i\leq t$, of the {\em nonsingular\/}
matrix
\begin{equation}
\label{eq:7}
\boldsymbol{P}(t):=
\left(\begin{smallmatrix}
P^0\\
P^1\\
P^2\vspace{-2mm}\\
\vdots\\
P^{t-3}\\
P^{t-2}\\
P^{t-1}
\end{smallmatrix}\right)\in\{1,-1\}^{t\times t}
\end{equation}
with entries~{\rm(\ref{eq:8})} constitute a {\em distinguished} sequence of certain $t$ topes of the oriented matroid~$\mathcal{H}$ for which, on the one hand, we have
\begin{equation}
\label{eq:10}
\boldsymbol{P}(t)=\mathbf{M}(\boldsymbol{D})\cdot\mathbf{W}(\boldsymbol{D})^{-1}\; ,
\end{equation}
for {\em any\/} symmetric cycle $\boldsymbol{D}$ in the {\em hypercube graph\/} $\boldsymbol{H}(t,2)$ of {\em topes\/} of the oriented matroid~$\mathcal{H}$.

On the other hand, for each {\em row}~$\boldsymbol{P}(t)^i=:P^{i-1}$, $1\leq i\leq t$, of the matrix~{\rm(\ref{eq:7})} with entries~{\rm(\ref{eq:8})}, regarded as a {\em vertex} of the hypercube graph~$\boldsymbol{H}(t,2)$ on the vertex set~$\{1,-1\}^t$,  there exists a {\em unique\/} row vector $\boldsymbol{x}^{i-1}:=\boldsymbol{x}^{i-1}(P^{i-1})$ $:=\boldsymbol{x}^{i-1}(P^{i-1},\boldsymbol{D})\in\{-1,0,1\}^t$, with $|\supp(\boldsymbol{x}^{i-1})|$ {\em odd}, such
that
\begin{equation*}
P^{i-1}=\boldsymbol{x}^{i-1}\mathbf{M}(\boldsymbol{D})\; ,
\end{equation*}
see~{\rm\cite[\S{}11.1]{M-PROM}}. Thus, we have
\begin{equation}
\label{eq:11}
\boldsymbol{P}(t)=
\underbrace{
\left(\begin{smallmatrix}
\boldsymbol{x}^0\\
\boldsymbol{x}^1\\
\boldsymbol{x}^2\vspace{-2mm}\\
\vdots\\
\boldsymbol{x}^{t-3}\\
\boldsymbol{x}^{t-2}\\
\boldsymbol{x}^{t-1}
\end{smallmatrix}\right)
}_{\mathbf{X}(\boldsymbol{P}(t),\boldsymbol{D})}\cdot\;\mathbf{M}(\boldsymbol{D})\; .
\end{equation}

Relations~{\rm(\ref{eq:10})} and~{\rm(\ref{eq:11})} yield
\begin{equation*}
\mathbf{M}(\boldsymbol{D})\cdot\mathbf{W}(\boldsymbol{D})^{-1}=\mathbf{X}(\boldsymbol{P}(t),\boldsymbol{D})\cdot\mathbf{M}(\boldsymbol{D})\; ,
\end{equation*}
for any symmetric cycle~$\boldsymbol{D}$ of the hypercube graph~$\boldsymbol{H}(t,2)$.
\end{remark}

\section{Vertex decompositions and edge decompositions in tope graphs with respect to the edges of their symmetric cycles}

In this section we discuss
(de)composition constructions related to the edge sequences of symmetric cycles in the tope graphs of simple oriented matroids. The results and their proofs are accompanied by a detailed example.

\begin{proposition}
\label{prop:1}
Let $\boldsymbol{D}:=(D^0,D^1,\ldots,D^{2t-1},D^0)$ be a symmetric cycle in the tope graph of a simple oriented matroid $\mathcal{M}:=(E_t,\mathcal{T})$, where $t$ is {\em even}. Let $\mathrm{S}(\boldsymbol{D})$ be the corresponding sequence of subtopes defined by~{\rm(\ref{eq:2})}.
\begin{itemize}
\item[\rm(i)] If $T\in\mathcal{T}$ is a {\em tope} of $\mathcal{M}$, then there exist a {\em unique} subset of subtopes $\overline{\boldsymbol{Q}}(T,\mathrm{S}(\boldsymbol{D}))\subset\mathrm{S}(\boldsymbol{D})$ and a {\em unique} set of {\em integers}~$\{\lambda_{Q'}\colon$ $Q'\in\overline{\boldsymbol{Q}}(T,\mathrm{S}(\boldsymbol{D}))\}$ such that
    \begin{equation}
    \label{eq:5}
    \sum_{Q'\in\overline{\boldsymbol{Q}}(T,\mathrm{S}(\boldsymbol{D}))}\lambda_{Q'}\cdot Q'=T\; ,
    \end{equation}
    where
    \begin{equation*}
    |\;\overline{\boldsymbol{Q}}(T,\mathrm{S}(\boldsymbol{D}))\;|=t\; ,
    \end{equation*}
and
    \begin{equation*}
    1\leq\lambda_{Q'}\leq t-1\; ,\ \ \ \ \text{and\ \ \ \ $\lambda_{Q'}$ are all {\em odd}}\; .
    \end{equation*}

    \item[\rm(ii)] If $S\in\mathcal{S}$ is a {\em subtope} of $\mathcal{M}$, then there exist a {\em unique inclusion-minimal} subset of subtopes $\overline{\boldsymbol{Q}}(S,\mathrm{S}(\boldsymbol{D}))\subset\mathrm{S}(\boldsymbol{D})$ and a {\em unique} set of {\em integers}~$\{\lambda_{Q'}\colon Q'\in\overline{\boldsymbol{Q}}(S,\mathrm{S}(\boldsymbol{D}))\}$ such that
    \begin{equation}
    \label{eq:6}
    \sum_{Q'\in\overline{\boldsymbol{Q}}(S,\mathrm{S}(\boldsymbol{D}))}\lambda_{Q'}\cdot Q'=S\; ,
    \end{equation}
    where
    \begin{equation*}
    1\leq\lambda_{Q'}\leq t-1\; .
    \end{equation*}
\end{itemize}
\end{proposition}

\begin{proof}
{\rm (i)} See Example~\ref{prop:2}(i).

Let $\boldsymbol{x}:=\boldsymbol{x}(T):=\boldsymbol{x}(T,\boldsymbol{D}):=(x_1,\ldots,x_t)\in\{-1,0,1\}^t$ be the unique
row vector such that
\begin{equation*}
T=\boldsymbol{x}\mathbf{M}(\boldsymbol{D})\; ,
\end{equation*}
where $|\supp(\boldsymbol{x})|$ is {\em odd}; see~\cite[\S{}11.1]{M-PROM}. In other words,
\begin{equation}
\label{eq:14}
T=\boldsymbol{x}\boldsymbol{P}(t)\mathbf{W}(\boldsymbol{D})=:\overline{\boldsymbol{x}}\,\mathbf{W}(\boldsymbol{D})\; ,
\end{equation}
where the vector $\overline{\boldsymbol{x}}:=\overline{\boldsymbol{x}}(T):=\overline{\boldsymbol{x}}(T,\mathrm{S}(\boldsymbol{D})):=
(\overline{x}_1,\ldots,\overline{x}_t)\in\mathbb{Z}^t$ is defined by
\begin{equation}
\label{eq:15}
\overline{\boldsymbol{x}}:=\boldsymbol{x}\boldsymbol{P}(t)\; .
\end{equation}
Since the entries of the matrix $\boldsymbol{P}(t)$ belong to the set~$\{1,-1\}$, the components of $\overline{\boldsymbol{x}}$ are all {\em odd\/} integers and, as a consequence,
$|\supp(\overline{\boldsymbol{x}})|=t$. We have
\begin{align*}
\overline{\boldsymbol{Q}}(T,\mathrm{S}(\boldsymbol{D}))&=
\{\sign(\overline{x}_i)\cdot S^{i-1}\colon 1\leq i\leq t\}\; ,\\
Q'&\in\{-S^{j-1},S^{j-1}\}\ \ \ \Longrightarrow\ \ \ \lambda_{Q'}=|\overline{x}_j|\; ,\ \ \ 1\leq j\leq t\; .
\end{align*}

{\rm (ii)} See Example~\ref{prop:2}(ii).

Let $T'$ and $T''$ be the two topes of $\mathcal{M}$ such that $S=\tfrac{1}{2}(T'+T'')$. If $\boldsymbol{x}_{T'}:=\boldsymbol{x}(T')\in\{-1,0,1\}^t$ and $\boldsymbol{x}_{T''}:=\boldsymbol{x}(T'')\in\{-1,0,1\}^t$ are the unique row vectors such that $T'=\boldsymbol{x}_{T'}\mathbf{M}(\boldsymbol{D})$ and $T''=\boldsymbol{x}_{T''}\mathbf{M}(\boldsymbol{D})$,
then we have
\begin{equation*}
\begin{split}
S&=\tfrac{1}{2}\bigl(\boldsymbol{x}_{T'}+\boldsymbol{x}_{T''}\bigr)\mathbf{M}(\boldsymbol{D})\\
&=\tfrac{1}{2}\bigl(\boldsymbol{x}_{T'}+\boldsymbol{x}_{T''}\bigr)\boldsymbol{P}(t)\mathbf{W}(\boldsymbol{D})
\; ,
\end{split}
\end{equation*}
or
\begin{equation*}
S=\overline{\boldsymbol{x}}\mathbf{W}(\boldsymbol{D})\; ,
\end{equation*}
where the vector $\overline{\boldsymbol{x}}:=\overline{\boldsymbol{x}}(S):=\overline{\boldsymbol{x}}(S,\mathrm{S}(\boldsymbol{D})):=
(\overline{x}_1,\ldots,\overline{x}_t)\in\mathbb{Z}^t$ is the vector
\begin{equation*}
\overline{\boldsymbol{x}}:=\tfrac{1}{2}\bigl(\overline{\boldsymbol{x}}_{T'}+\overline{\boldsymbol{x}}_{T''}\bigr)
=\tfrac{1}{2}\bigl(\boldsymbol{x}_{T'}+\boldsymbol{x}_{T''}\bigr)\boldsymbol{P}(t)\; .
\end{equation*}

Thus,
\begin{align*}
\overline{\boldsymbol{Q}}(S,\mathrm{S}(\boldsymbol{D}))&=
\{\sign(\overline{x}_{T',i}+\overline{x}_{T'',i})\cdot S^{i-1}\colon \overline{x}_{T'',i}\neq-\overline{x}_{T',i},\; 1\leq i\leq t\}\; ,\\
Q'&\in\{-S^{j-1},S^{j-1}\}\ \ \ \Longrightarrow\ \ \ \lambda_{Q'}=\tfrac{1}{2}|\overline{x}_{T',j}+\overline{x}_{T'',j}|\; ,\ \ \ 1\leq j\leq t\; .
\end{align*}
\end{proof}

\begin{figure}
\center{
\begin{overpic}[scale=1.0]{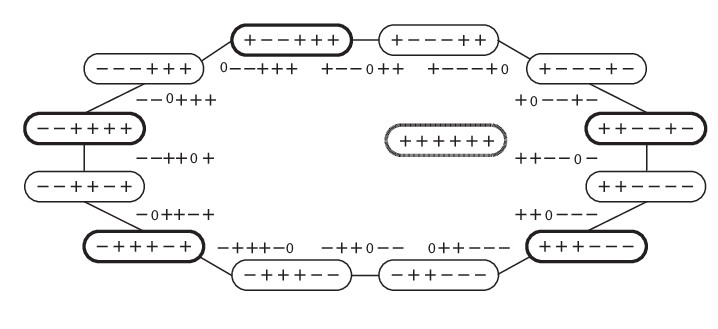}
\put(10,7.1){\makebox(0,0)[c]{$D^0=:Q^0$}}
\put(4,15.7){\makebox(0,0)[r]{$D^1$}}
\put(0.8,32){\makebox(0,0)[l]{$D^2=:Q^1$}}
\put(12,40){\makebox(0,0)[c]{$D^3$}}
\put(31,44.3){\makebox(0,0)[l]{$D^4=:Q^2$}}
\put(71,44){\makebox(0,0)[c]{$D^5$}}
\put(92.5,40.5){\makebox(0,0)[r]{$D^6$}}
\put(95,32){\makebox(0,0)[c]{$D^7=:Q^3$}}
\put(98,15.7){\makebox(0,0)[c]{$D^8$}}
\put(90,7.1){\makebox(0,0)[c]{$D^9=:Q^4$}}
\put(71,3.5){\makebox(0,0)[c]{$D^{10}$}}
\put(31,3.5){\makebox(0,0)[c]{$D^{11}$}}
\put(53,22){\makebox(0,0)[c]{$\textcolor{blue}{\mathrm{T}^{(+)}}$}}

\put(30.75,16){\makebox(0,0)[l]{$S^0$}}
\put(30.75,24){\makebox(0,0)[l]{$S^1$}}
\put(30.75,32){\makebox(0,0)[l]{$S^2$}}
\put(37,33){\makebox(0,0)[l]{$S^3$}}
\put(49,33){\makebox(0,0)[l]{$S^4$}}
\put(63,33){\makebox(0,0)[r]{$S^5$}}
\put(70.5,32){\makebox(0,0)[r]{$S^6$}}
\put(70,21){\makebox(0,0)[l]{$S^7$}}
\put(70.5,16){\makebox(0,0)[r]{$S^8$}}
\put(63,14){\makebox(0,0)[r]{$S^9$}}
\put(49,14){\makebox(0,0)[l]{$S^{10}$}}
\put(37,14){\makebox(0,0)[l]{$S^{11}$}}
\end{overpic}
}
\caption{A {\em symmetric cycle\/}
$\boldsymbol{D}:=(D^0,D^1,\ldots,D^{2t-1},D^0)$ in the {\em hypercube graph\/} $\boldsymbol{H}(t,2)$ on the vertex set~$\{1,-1\}^t$, where $t:=6$.
The edges of the cycle $\boldsymbol{D}$ are labeled by the corresponding {\em subtopes\/} of the oriented matroid $\mathcal{H}:=(E_6,\{1,-1\}^6)$; for example, the edge $\{D^0,D^1\}$ is labeled by the subtope $S^0:=(-1,0,1,1,-1,1)$. Here the {\em sign components\/} $1$, $-1$ and $0$ of {\em covectors\/} of $\mathcal{H}$ are substituted by the familiar symbols `$+$', `$-$' and~`0', respectively.
\newline
For the vertex~$\textcolor{blue}{\mathrm{T}^{(+)}}:=(1,1,1,1,1,1)$
of the graph~$\boldsymbol{H}(6,2)$, on the one hand, we have
$\textcolor{blue}{\mathrm{T}^{(+)}}=Q^0+Q^1+Q^2+Q^3$ $+\;Q^4$
for a {\em unique inclusion-minimal\/} subset~$\boldsymbol{Q}(\textcolor{blue}{\mathrm{T}^{(+)}},\boldsymbol{D})=:\{Q^0,\ldots,Q^4\}$ of five {\em topes\/} in the vertex sequence $\mathrm{V}(\boldsymbol{D})$ of the cycle~$\boldsymbol{D}$, where~$Q^0:=D^0$, $Q^1:=D^2$, $Q^2:=D^4$, $Q^3:=D^7$ and $Q^4:=D^9$; see~\cite[\S{}11.1]{M-PROM}.
\newline
On the other hand,~$\textcolor{blue}{\mathrm{T}^{(+)}}=S^1+S^2+5S^4+3S^6+3S^9+5S^{11}$ for the {\em unique\/}
set~$\overline{\boldsymbol{Q}}(\textcolor{blue}{\mathrm{T}^{(+)}},\mathrm{S}(\boldsymbol{D}))$ of
$t:=6$ {\em subtopes\/} $S^1$, $S^2$, $S^4$, $S^6$, $S^9$ and $S^{11}$ associated with edges of the cycle~$\boldsymbol{D}$, for which the corresponding {\em integer\/} coefficients are all {\em positive} and {\em odd}; see Proposition~\ref{prop:1}(i).
}
\label{fig:1}
\end{figure}
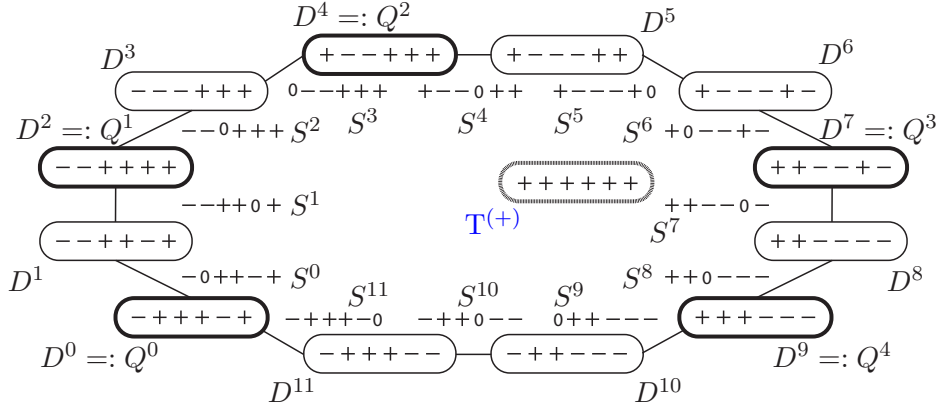

$\quad$

$\quad$

$\quad$

$\quad$

\newpage
\begin{example}
\label{prop:2}
Let $\boldsymbol{D}$ be a {\em symmetric cycle}, depicted in~Figure~{\rm\ref{fig:1}}, in the {\em hypercube graph\/} of {\em topes\/} of the oriented matroid $\mathcal{H}:=(E_6,\{1,-1\}^6)$.
The corresponding matrices $\mathbf{M}$ and $\mathbf{W}$ that describe {\em vertices\/} and {\em edges\/} of the cycle~$\boldsymbol{D}$ in the expression~{\rm(\ref{eq:9})} are as follows:
\begin{equation*}
\begin{split}
\mathbf{M}:=\mathbf{M}(\boldsymbol{D})=
\left(\begin{smallmatrix}
D^0\\
D^1\\
D^2\\
D^3\\
D^4\\
D^5
\end{smallmatrix}\right)=\left(\begin{smallmatrix}
-1&1&1&1&-1&1\\
-1&-1&1&1&-1&1\\
-1&-1&1&1&1&1\\
-1&-1&-1&1&1&1\\
1&-1&-1&1&1&1\\
1&-1&-1&-1&1&1
\end{smallmatrix}\right)\; ,\\
\mathbf{W}:=\mathbf{W}(\boldsymbol{D})=
\left(\begin{smallmatrix}
S^0\\
S^1\\
S^2\\
S^3\\
S^4\\
S^5
\end{smallmatrix}\right)=\left(\begin{smallmatrix}
-1&0&1&1&-1&1\\
-1&-1&1&1&0&1\\
-1&-1&0&1&1&1\\
0&-1&-1&1&1&1\\
1&-1&-1&0&1&1\\
1&-1&-1&-1&1&0
\end{smallmatrix}\right)\; .
\end{split}
\end{equation*}

\rm{(i)} Here we illustrate the proof of Proposition~\ref{prop:1}(i).

We would like to find the decomposition of the positive tope $\mathrm{T}^{(+)}$ $:=(1,1,1,1,1,1)$ with respect to
the set $\mathrm{S}(\boldsymbol{D})$ of subtopes~(\ref{eq:2}) associated with the edges of the cycle~$\boldsymbol{D}$; see~Figure~{\rm\ref{fig:1}}.

We have
\begin{equation*}
\boldsymbol{x}:=\boldsymbol{x}(\mathrm{T}^{(+)},\boldsymbol{D})=(1,-1,1,-1,1,0)
\end{equation*}
and
\begin{equation*}
\begin{split}
\overline{\boldsymbol{x}}:\!&=\overline{\boldsymbol{x}}(\mathrm{T}^{(+)},\mathrm{S}(\boldsymbol{D}))=\boldsymbol{x}\boldsymbol{P}(6)=
(1,-1,1,-1,1,0)\cdot
\left(\begin{smallmatrix}
1&-1&1&-1&1&-1\\
1&1&-1&1&-1&1\\
-1&1&1&-1&1&-1\\
1&-1&1&1&-1&1\\
-1&1&-1&1&1&-1\\
1&-1&1&-1&1&1
\end{smallmatrix}\right)\\&=(-3,1,1,-3,5,-5)\; .
\end{split}
\end{equation*}
Thus, we have
\begin{equation*}
\begin{split}
\overline{\boldsymbol{Q}}(\mathrm{T}^{(+)},\mathrm{S}(\boldsymbol{D}))&=
\{\sign(\overline{x}_i)\cdot S^{i-1}\colon 1\leq i\leq 6\}=
\{\underbrace{-S^0}_{S^6},S^1,S^2,\underbrace{-S^3}_{S^9},S^4,\underbrace{-S^5}_{S^{11}}\}\\&=\{S^1,S^2,S^4,S^6,S^9,S^{11}\}\; ,
\end{split}
\end{equation*}
that is,
\begin{equation*}
\begin{split}
\mathrm{T}^{(+)}:\!&=(1,1,1,1,1,1)\\&=
|\overline{x}_2|\cdot S^1+|\overline{x}_3|\cdot S^2+|\overline{x}_5|\cdot S^4+|\overline{x}_1|\cdot S^6+
|\overline{x}_4|\cdot S^9+|\overline{x}_6|\cdot S^{11}
\\&=
S^1 + S^2 + 5S^4 + 3S^6 + 3S^9 + 5S^{11}
\\&=
(-1,-1,1,1,0,1)+(-1,-1,0,1,1,1)+5\cdot(1,-1,-1,0,1,1)
\\ &\phantom{=}+3\cdot(1,0,-1,-1,1,-1)+3\cdot(0,1,1,-1,-1,-1)+5\cdot(-1,1,1,1,-1,0)\; ,
\end{split}
\end{equation*}
for the {\em unique\/} set~$\overline{\boldsymbol{Q}}(\mathrm{T}^{(+)},\mathrm{S}(\boldsymbol{D}))=\{S^1,S^2,S^4,S^6,S^9,S^{11}\}$
of $6$ subtopes associated with edges of the cycle~$\boldsymbol{D}$ for which the corresponding {\em integer\/} coefficients are all {\em positive} and {\em odd}.

\newpage
\rm{(ii)} Let us now illustrate the proof of Proposition~\ref{prop:1}(ii).

Consider the subtope
\begin{equation*}
S:=(-1,-1,0,1,-1,-1)
\end{equation*}
of the oriented matroid~$\mathcal{H}:=(E_6,\{1,-1\}^6)$ that corresponds to an edge~$\{T',T''\}$ of its hypercube graph of topes $\boldsymbol{H}(6,2)$, where
\begin{equation*}
T':=(-1,-1,1,1,-1,-1)\ \ \ \text{and}\ \ \ T'':=(-1,-1,-1,1,-1,-1)\; .
\end{equation*}
We have
\begin{equation*}
\begin{split}
S&=\underbrace{\tfrac{1}{2}\bigl(\overline{\boldsymbol{x}}_{T'}+\overline{\boldsymbol{x}}_{T''}\bigr)}_{\overline{\boldsymbol{x}}(S)}
\mathbf{W}(\boldsymbol{D})\\
&=\underbrace{\tfrac{1}{2}\bigl(\boldsymbol{x}_{T'}+\boldsymbol{x}_{T''}\bigr)
\boldsymbol{P}(6)}_{\overline{\boldsymbol{x}}(S)}\mathbf{W}(\boldsymbol{D})\; .
\end{split}
\end{equation*}
Since $\boldsymbol{x}_{T'}:=\boldsymbol{x}(T')=(-1,1,0,0,0,-1)$ and $\boldsymbol{x}_{T''}:=\boldsymbol{x}(T'')=(-1,1,-1,1,0,-1)$,
we have
\begin{equation*}
\begin{split}
\overline{\boldsymbol{x}}(S)&=\tfrac{1}{2}\bigl((-1,1,0,0,0,-1)+(-1,1,-1,1,0,-1)\bigr)\cdot\!
\left(\begin{smallmatrix}
1&-1&1&-1&1&-1\\
1&1&-1&1&-1&1\\
-1&1&1&-1&1&-1\\
1&-1&1&1&-1&1\\
-1&1&-1&1&1&-1\\
1&-1&1&-1&1&1
\end{smallmatrix}\right)\\&=
(-1,1,-\tfrac{1}{2},\tfrac{1}{2},0,-1)\cdot\!
\left(\begin{smallmatrix}
1&-1&1&-1&1&-1\\
1&1&-1&1&-1&1\\
-1&1&1&-1&1&-1\\
1&-1&1&1&-1&1\\
-1&1&-1&1&1&-1\\
1&-1&1&-1&1&1
\end{smallmatrix}\right)\\&=(0,2,-3,4,-4,2)\; .
\end{split}
\end{equation*}
We see that
\begin{equation*}
S:=(-1,-1,0,1,-1,-1)=2S^1-3\cdot\underbrace{S^2}_{-S^8}+4S^3-4\cdot\underbrace{S^4}_{-S^{10}}+2S^5
\end{equation*}
or, in other words,
\begin{equation*}
\begin{split}
S:\!&=(-1,-1,0,1,-1,-1)\\&=2S^1+4S^3+2S^5+3S^8+4S^{10}\\&=
2\cdot (-1,-1,1,1,0,1)+4\cdot (0,-1,-1,1,1,1)+2\cdot (1,-1,-1,-1,1,0)
\\&\phantom{=}+3\cdot (1,1,0,-1,-1,-1)+4\cdot (-1,1,1,0,-1,-1)\; ,
\end{split}
\end{equation*}
for the {\em unique inclusion-minimal\/} set~$\overline{\boldsymbol{Q}}(S,\mathrm{S}(\boldsymbol{D}))=\{S^1,S^3,S^5,S^8,S^{10}\}$
of subtopes associated with edges of the cycle~$\boldsymbol{D}$ for which the corresponding {\em integer\/} coefficients are all {\em positive}.
\end{example}

\section{Vertex decompositions in hypercube graphs with respect to the edges of their distinguished symmetric cycles}

Let $\boldsymbol{R}$ be a distinguished symmetric cycle in the {\em hypercube graph\/}~$\boldsymbol{H}(t,2)$ of {\em topes\/} of the oriented
matroid~$\mathcal{H}:=(E_t,\{1,-1\}^t)$, where $t$ is {\em even}, defined as follows:
\begin{equation}
\label{eq:12}
\begin{split}
R^0:\!&=\mathrm{T}^{(+)}\; ,\\
R^s:\!&={}_{-[s]}R^0\; ,\ \ \ 1\leq s\leq t-1\; ,
\end{split}
\end{equation}
and
\begin{equation}
\label{eq:13}
R^{k+t}:=-R^k\; ,\ \ \ 0\leq k\leq t-1\; .
\end{equation}
No matter how large the dimension $t$ of the {\em discrete hypercube\/} $\{1,-1\}^t$ is, four assertions of~\cite[Prop.~2.4]{M-SC-II} allow us to find the {\em linear algebraic\/}
decompositions
\begin{equation*}
T=\sum_{Q\in\boldsymbol{Q}(T,\boldsymbol{R})}Q
\end{equation*}
of vertices $T$ of the graph~$\boldsymbol{H}(t,2)$ on the vertex set~$\{1,-1\}^t$, by means of {\em inclusion-minimal\/} subsets~$\boldsymbol{Q}(T,\boldsymbol{R})\subset\mathrm{V}(\boldsymbol{R})$ of {\em odd\/} cardinality, in an {\em explicit\/} and {\em computation-free\/} way.

We now present a (subtope)
companion to~\cite[Prop.~2.4]{M-SC-II}. As earlier, $\boldsymbol{P}(t)^s$, $1\leq s\leq t$, denotes the $s$th row $P^{s-1}$ of the matrix $\boldsymbol{P}(t)$ given in~(\ref{eq:7}), with entries~\rm(\ref{eq:8}).

\begin{proposition}
Let $\boldsymbol{R}$ be the distinguished symmetric cycle, defined
by {\rm(\ref{eq:12})(\ref{eq:13})}, in the hypercube graph~$\boldsymbol{H}(t,2)$ on the vertex set~$\{1,-1\}^t$, where~$t$ is {\em even}.

Let $A$ be a nonempty subset of the ground set $E_t$, viewed as a {\em disjoint} union
\begin{equation*}
A=[i_1,j_1]\;\dot\cup\;[i_2,j_2]\;\dot\cup\;\cdots\;\dot\cup\;[i_{\varrho},j_{\varrho}]
\end{equation*}
of intervals of $E_t$, such that
\begin{equation*}
j_1+2\leq i_2,\ \ j_2+2\leq i_3,\ \ \ldots,\ \
j_{\varrho-1}+2\leq i_{\varrho}\; ,
\end{equation*}
for some $\varrho:=\varrho(A)$.

Consider the vector $\overline{\boldsymbol{x}}({}_{-A}\mathrm{T}^{(+)},\mathrm{S}(\boldsymbol{R}))\in\mathbb{Z}^t$ defined by
\begin{equation*}
{}_{-A}\mathrm{T}^{(+)}=:\overline{\boldsymbol{x}}({}_{-A}\mathrm{T}^{(+)},\mathrm{S}(\boldsymbol{R}))\cdot\mathbf{W}(\boldsymbol{R})\; ,
\end{equation*}
cf.~{\rm(\ref{eq:14})} and {\rm(\ref{eq:15})}.
\begin{itemize}
\item[\rm(i)]
If $\{1,t\}\cap A=\{1\}$, then
\begin{equation*}
\overline{\boldsymbol{x}}({}_{-A}\mathrm{T}^{(+)},\mathrm{S}(\boldsymbol{R}))=
\sum_{1\leq k\leq\varrho}\boldsymbol{P}(t)^{j_k+1}\ \ - \ \ \sum_{2\leq \ell\leq\varrho}\boldsymbol{P}(t)^{i_{\ell}}\; ,
\end{equation*}
that is, for a component $\overline{x}_e$ of this vector, where $e\in E_t$, we have
\begin{multline*}
\overline{x}_e({}_{-A}\mathrm{T}^{(+)},\mathrm{S}(\boldsymbol{R}))=
\sum_{1\leq k\leq\varrho}
\begin{cases}
(-1)^{e+j_k+1}\; , & \text{if $j_k<e$\; ,}\\
(-1)^{e+j_k}\; , & \text{if $j_k\geq e$}
\end{cases}\\
-\sum_{2\leq \ell\leq\varrho}
\begin{cases}
(-1)^{e+i_{\ell}}\; , & \text{if $i_{\ell}\leq e$\; ,}\\
(-1)^{e+i_{\ell}+1}\; , & \text{if $i_{\ell}>e$\; .}
\end{cases}
\end{multline*}

\item[\rm(ii)]
If $\{1,t\}\cap A=\{1,t\}$, then
\begin{equation*}
\overline{\boldsymbol{x}}({}_{-A}\mathrm{T}^{(+)},\mathrm{S}(\boldsymbol{R}))=-\boldsymbol{P}(t)^1\  + \  \sum_{1\leq k\leq\varrho-1}\boldsymbol{P}(t)^{j_k+1}\ \ - \ \
\sum_{2\leq\ell \leq\varrho}\boldsymbol{P}(t)^{i_{\ell}}\; ,
\end{equation*}
that is, for a component $\overline{x}_e$ of this vector we have
\begin{multline*}
\overline{x}_e({}_{-A}\mathrm{T}^{(+)},\mathrm{S}(\boldsymbol{R}))=
(-1)^e\ \ +\ \
\sum_{1\leq k\leq\varrho-1}
\begin{cases}
(-1)^{e+j_k+1}\; , & \text{if $j_k<e$\; ,}\\
(-1)^{e+j_k}\; , & \text{if $j_k\geq e$}
\end{cases}\\
-\sum_{2\leq \ell\leq\varrho}
\begin{cases}
(-1)^{e+i_{\ell}}\; , & \text{if $i_{\ell}\leq e$\; ,}\\
(-1)^{e+i_{\ell}+1}\; , & \text{if $i_{\ell}>e$\; .}
\end{cases}
\end{multline*}

\item[\rm(iii)]
If $|\{1,t\}\cap A|=0$, then
\begin{equation*}
\overline{\boldsymbol{x}}({}_{-A}\mathrm{T}^{(+)},\mathrm{S}(\boldsymbol{R}))=\boldsymbol{P}(t)^1\  + \  \sum_{1\leq k\leq\varrho}\boldsymbol{P}(t)^{j_k+1}\  - \
\sum_{1\leq\ell \leq\varrho}\boldsymbol{P}(t)^{i_{\ell}}\; ,
\end{equation*}
that is, for a component $\overline{x}_e$ of this vector we have
\begin{multline*}
\overline{x}_e({}_{-A}\mathrm{T}^{(+)},\mathrm{S}(\boldsymbol{R}))=
(-1)^{e+1}\ \ +\ \
\sum_{1\leq k\leq\varrho}
\begin{cases}
(-1)^{e+j_k+1}\; , & \text{if $j_k<e$\; ,}\\
(-1)^{e+j_k}\; , & \text{if $j_k\geq e$}
\end{cases}\\
-\sum_{1\leq \ell\leq\varrho}
\begin{cases}
(-1)^{e+i_{\ell}}\; , & \text{if $i_{\ell}\leq e$\; ,}\\
(-1)^{e+i_{\ell}+1}\; , & \text{if $i_{\ell}>e$\; .}
\end{cases}
\end{multline*}

\item[\rm(iv)]
If $\{1,t\}\cap A=\{t\}$, then
\begin{equation*}
\overline{\boldsymbol{x}}({}_{-A}\mathrm{T}^{(+)},\mathrm{S}(\boldsymbol{R}))=
\sum_{1\leq k\leq\varrho-1}\boldsymbol{P}(t)^{j_k+1}\ \ - \ \ \sum_{1\leq \ell\leq\varrho}\boldsymbol{P}(t)^{i_{\ell}}\; ,
\end{equation*}
that is, for a component $\overline{x}_e$ of this vector we have
\begin{multline*}
\overline{x}_e({}_{-A}\mathrm{T}^{(+)},\mathrm{S}(\boldsymbol{R}))=
\sum_{1\leq k\leq\varrho-1}
\begin{cases}
(-1)^{e+j_k+1}\; , & \text{if $j_k< e$\; ,}\\
(-1)^{e+j_k}\; , & \text{if $j_k\geq e$}
\end{cases}\\
-\sum_{1\leq \ell\leq\varrho}
\begin{cases}
(-1)^{e+i_{\ell}}\; , & \text{if $i_{\ell}\leq e$\; ,}\\
(-1)^{e+i_{\ell}+1}\; , & \text{if $i_{\ell}>e$\; .}
\end{cases}
\end{multline*}
\end{itemize}
\end{proposition}

In particular, we have
\begin{equation*}
1\leq j<t\ \ \ \Longrightarrow\ \ \
\overline{\boldsymbol{x}}({}_{-[j]}\mathrm{T}^{(+)},\mathrm{S}(\boldsymbol{R}))=\boldsymbol{P}(t)^{j+1}\; ;
\end{equation*}
\begin{equation*}
\overline{\boldsymbol{x}}(\mathrm{T}^{(-)},\mathrm{S}(\boldsymbol{R}))=
-\overline{\boldsymbol{x}}(\mathrm{T}^{(+)},\mathrm{S}(\boldsymbol{R}))=-\boldsymbol{P}(t)^1\; ;
\end{equation*}
\begin{equation*}
1<i<j<t\ \ \ \Longrightarrow\ \ \
\overline{\boldsymbol{x}}({}_{-[i,j]}\mathrm{T}^{(+)},\mathrm{S}(\boldsymbol{R}))=\boldsymbol{P}(t)^1-\boldsymbol{P}(t)^i+\boldsymbol{P}(t)^{j+1}\; ;
\end{equation*}
\begin{equation*}
1<i\leq t\ \ \ \Longrightarrow\ \ \
\overline{\boldsymbol{x}}({}_{-[i,t]}\mathrm{T}^{(+)},\mathrm{S}(\boldsymbol{R}))=-\boldsymbol{P}(t)^i\; .
\end{equation*}
If $s\in E_t$, then for a row vector $\overline{\boldsymbol{y}}(s):=\overline{\boldsymbol{y}}(s;t)$ defined by
\begin{equation*}
\overline{\boldsymbol{y}}(s):=\overline{\boldsymbol{x}}({}_{-s}\mathrm{T}^{(+)},\mathrm{S}(\boldsymbol{R}))\; ,
\end{equation*}
we have
\begin{equation*}
\overline{\boldsymbol{y}}(s)=
\begin{cases}
\phantom{-}\boldsymbol{P}(t)^2\; , & \text{if $s=1$},\\
\phantom{-}\boldsymbol{P}(t)^1-\boldsymbol{P}(t)^s+\boldsymbol{P}(t)^{s+1}\; , & \text{if $1<s<t$},\\
-\boldsymbol{P}(t)^t\; , & \text{if $s=t$}.
\end{cases}
\end{equation*}

\vspace{5mm}
\end{document}